\documentclass[11pt]{amsart}

\newcounter{cnt1}
\newcounter{cnt2}
\newcommand{\blr}{\begin{list}{$(\roman{cnt1})$}
    {\usecounter{cnt1} \setlength{\topsep}{0pt}
        \setlength{\itemsep}{0pt}}}
\newcommand{\bla}{\begin{list}{$($\alph{cnt2}$)$}
    {\usecounter{cnt2} \setlength{\topsep}{0pt}
        \setlength{\itemsep}{0pt}}}
\newcommand{\el}{\end{list}}
\newtheorem{thm}{Theorem}
\newtheorem{lem}[thm]{Lemma}
\newtheorem{cor}[thm]{Corollary}

{\newtheorem{Def}[thm]{Definition}

\newtheorem{rem}[thm]{Remark}
\newcommand{\Rem}{\begin{rem} \rm}
\newcommand{\bdfn}{\begin{Def} \rm}
\newcommand{\edfn}{\end{Def}}

\begin{document}
\large
\title[Compact operators]{ Best approximation in spaces of compact operators }
\author[Rao]{T. S. S. R. K. Rao}
\address[T. S. S. R. K. Rao]
{Department of Mathematics\\
Ashoka University\\
Rai\\ India,
\textit{E-mail~:}
\textit{srin@fulbrightmail.org}}
  \maketitle
\begin{abstract}
Let ${\mathcal K}(X,Y)$ be the space of compact operators. For a proximinal subspace $Z \subset Y$, this paper deals with the question, when does every $Y$-valued compact operator admits a  $Z$-valued compact best approximation?  For any reflexive Banach space  $X$ and for a $L^1$-predual space $Y$, if $Z \subset Y$ is a strongly proximinal subspace of finite codimension, we show that ${\mathcal K}(X,Z)$ is a proximinal subspace of ${\mathcal K}(X,Y)$ under an additional condition on the position of ${\mathcal K}(X,Z)$. When $Y$ is a $c_0$-direct sum of finite dimensional spaces we achieve a strong transitivity result by showing that for any proximinal subspace of finite codimension
$Z \subset Y$,  every $Y$-valued bounded operator admits a best $Z$-valued compact approximation.
\end{abstract}
\section { Introduction}
Let $Y$ be a real Banach space. We recall that a closed subspace $Z\subset Y$, is said to be proximinal, if
for any $y\in Y$, there is a $z_0 \in Z$ such that $d(y,Z)= \|y-z_0\|$. For a proximinal subspace $Z \subset Y$, an interesting question in operator theory is to decide if $Z$-valued  operators form a proximinal subspace of the space of $Y$-valued  operators, for a fixed domain space $X$. In a dual space, if  $Z \subset Y^\ast$ is a weak$^\ast$-closed subspace, then using the weak$^\ast$-operator topology, one can see that the space of operators ${\mathcal L}(X,Z) \subset {\mathcal L}(X,Y^\ast)$ is a closed (weak$^\ast$-operator topology) and hence proximinal subspace. In spaces of compact operators, not many instances are known where best approximation can be achieved for all domain spaces.
\vskip 2em
For a compact Hausdorff space $K$, let $C(K,X)$ denote the space of continuous $X$-valued  functions, equipped with the supremum norm.
Using the canonical identification of $C(K, X^\ast)$ with ${\mathcal K}(X,C(K))$, it is known that if $H \subset X^\ast$ is a finite dimensional Chebyshev space(best approximations are unique), then ${\mathcal K}(H^\ast,C(K))$ is a proximinal subspace of ${\mathcal K}(X,C(K))$ ( \cite{LC} Corollary 2.2).
\vskip 2em
On the other hand for any finite dimensional subspace $H \subset L^1([0,1])$, keeping the canonical embedding of $L^1([0,1]) \subset L^\infty([0,1])^\ast$ in mind,  considering the inclusions,
  $${\mathcal K}(H^\ast,C([0,1])) = C([0,1],H) \subset C([0,1],L^1([0,1]))$$
  $$\subset {\mathcal K}(L^\infty([0,1]), C([0,1]).$$
we have that ${\mathcal K}(H^\ast,C([0,1])$ is not a proximinal subspace of
$${\mathcal K}(L^\infty([0,1], C([0,1])),$$ as $C([0,1],H)$ is not a proximinal subspace of $C([0,1],L^1([0,1])$. See \cite{LC} Theorem 2.6.
\vskip 2em
Not many results where proximinality is preserved among spaces of operators are known when one considers infinite dimensional subspaces. Let $X$ be a reflexive Banach space. Let $Y$ be a Banach space such that $Y^\ast$ is isometric to a $L^1(\mu)$ for a positive measure $\mu$. Such spaces are called Lindenstrauss spaces or $L^1$-predual spaces. These were extensively studied by J. Lindenstrauss (see \cite{L}, Chapter 7).
\vskip 2em
For any compact space $K$, $C(K)$ is a $L^1$-predual space. Also  $X$ is a $L^1$-predual space if and only if $X^{\ast\ast}$ is a $C(K)$-space.  Motivated by the above classical example, we consider the question, for a $L^1$-predual space $Y$ and for a finite codimensional proximinal subspace $Z \subset Y$, when is ${\mathcal K}(X,Z) \subset {\mathcal K}(X,Y)$ a proximinal subspace for all reflexive spaces $X$? We partially answer the question based on  a stronger notion of proximinality  (see Section 2) and the structure of ${\mathcal K}(X,Z)$.
\vskip 1em
We consider a discrete variation of $L^1$-predual spaces by considering Banach spaces $Y$ for which $Y^\ast = \bigoplus_1 Y_i$, for an infinite family $\{Y_i\}_{i \in I}$ of finite dimensional spaces. Here we achieve a stronger result by showing that for any reflexive Banach space $X$ with the approximation property and
for any finite codimensional strongly proximinal subspace $Z \subset Y$, ${\mathcal K}(X,Z)$ is a  proximinal subspace of ${\mathcal L}(X,Y)$.
\vskip 1em
The author thanks the communicating Editors and the referees, for their efficient handling of this article during the pandemic (submitted during the first wave and a corrected version submitted during the harsh second wave in India).
\section{Strongly proximinal subspaces}
In order to achieve best compact approximation, we need a stronger notion of proximinality. We recall an equivalent version of strong proximinality from
\cite{GI}.
For a proximinal subspace $Y \subset X$, for $x \in X$, let $P(x) = \{y \in Y: d(x,Y) = \|x-y\|\}$ denote the set of best approximants.
\vskip 1em
\begin{Def}
	
	A proximinal subspace $Y \subset X$ is said to be strongly proximinal, if for any $x \in X$, for any minimising sequence $\{y_n\}_{n \geq 1} \subset Y$, i.e, $d(x,Y) = \lim_n \|x-y_n\|$, there is a subsequence $\{y_{n_k}\}$ and a sequence $\{z_k\}_{k \geq 1} \subset P(x)$ such that $\|y_{n_k}-z_k\| \rightarrow 0$ as $k \rightarrow \infty$.
	
\end{Def}
\vskip 1em
\begin{rem}It is easy to see that any finite dimensional subspace is strongly proximinal. It was shown in
 \cite{R} that for a strongly proximinal subspace $Y \subset X$ and for any finite dimensional space $F \subset X$, $F+Y$ is a strongly proximinal subspace.
\end{rem}
In order to give a necessary condition for a closed subspace of finite codimension to be strongly proximinal, we recall the notion of strong subdifferentiability from (\cite{FP}).
\begin{Def} A non-zero vector, $x \in X$  is said to be a point of strong subdifferentiability  (in short,  $SSD$ point) if the one sided limit
 $\lim_{t \rightarrow 0^{+}} \frac{\|x+th\|-\|x\|}{t}$ exists uniformly for  $h\in B_X$, where $B_X$ denotes the closed unit ball of the Banach space $X$.
 \end{Def}
If $x^\ast \in X^\ast$ is a $SSD$ point, it was shown in \cite{FP},( Theorem 3.3) that $x^\ast$ attains its norm. Since any extreme point of the unit ball of $L^1(\mu)$ is given by $\pm \frac{\chi_{A}}{\mu(A)}$ for a $\mu$-atom $A$, it is easy to see that it is a $SSD$ point.
\vskip 1em
 We recall an interesting relation between these concepts, a geometric statement from \cite{GI} (Lemma1. 1 and Theorem 2.5) that in a dual space $X^\ast$, a non-zero vector  $x^\ast \in X^\ast$
 is a $SSD$ point of  $X^*$ if and only if $\ker(x^\ast)$ is a strongly proximinal subspace of $X$.
\begin{rem}
We note that for any strongly proximinal subspace of finite codimension $Y \subset X$, if $Z$ is a closed subspace such that $Y \subset Z \subset X$, then since $Z = F+Y$ for some finite dimensional space $F \subset X$, by Remark 2, $Z$ is also a strongly proximinal subspace of $X$. In particular for such a $Y$, $Y^\bot$ consists of $SSD$ points of $X^\ast$.
	
\end{rem}
We next state a result from (\cite{JR}) that describes $SSD$ points of a dual $L^1(\mu)$ space. For a $L^1$-predual space $X$, it is easy to see that if
$f_1,...,f_n$ are linearly independent extreme points of the dual unit ball, as each $f_i$ corresponds to a normalized characteristic function of a measure atom of the underlying measure space, the span of $\{f_1,...,f_n\}$ is isometric to  the discrete space $\ell^1(n)$ and $X^\ast = L^1(\mu) = \ell^1(n) \bigoplus_1 L^1(\nu)$, where $\nu$ lives on the complement of the union of the supports of the $f_i$'s.
\begin{lem}
	Let $X$ be a $L^1$ predual space. Any $SSD$ point of $X^\ast$ is contained in $F = span\{f_1,...,f_k\}$ for some finitely many independent vectors $f_1,...,f_k$ which are extreme points of the unit ball of $X^\ast$. Consequently $F$ is isometric to $\ell^1(k)$ and $X^\ast = F \bigoplus_1 N$ ($\ell^1$-sum) for some closed subspace $N \subset X^\ast$.
\end{lem}
Let $X$ be a reflexive Banach space with the metric approximation property. Let $I$ be an infinite index set. Let $\{Y_i\}_{i \in I}$ be a family of finite dimensional Banach spaces. Let $Y$ be a Banach space such that $Y^\ast = \bigoplus_1 Y_i$. Let $Z \subset {\mathcal K}(X,Y)$ be a closed subspace such that $Z^\bot$ consists of $SSD$  points of ${\mathcal K}(X,Y)^\ast$. We show that $Z$ is a proximinal, factor reflexive subspace of ${\mathcal K}(X,Y)$.
We could achieve proximinality in ${\mathcal L}(X,Y)$ when $Y = \bigoplus_{c_0} Y_i^\ast$.  Let $Z \subset Y$ be a proximinal subspace of finite codimension. ${\mathcal K}(X,Z)$ is a proximinal subspace of
${\mathcal L}(X,Y)$.

\vskip 1em
\section{Main Results}
In order to study the compact approximation problem, we use the identification of space of compact operators as injective tensor product spaces. The monographs by Ryan (\cite{RR}), Diestel and Uhl (\cite{DU}, Chapter VIII) are standard references.
\vskip 1em
We will be using two classical results from the theory of best approximations from closed subspaces. One is a classical result of Garkavi (\cite{S} Chapter 1, Theorem 2.1) that says, if $Z_1 \subset Z_2 \subset Y$ are closed subspaces such that $Y/Z_1$ is reflexive and $Z_1$ is a proximinal subspace of $Y$, then $Z_2$ is also a proximinal subspace of $Y$.
\vskip 1em
The second result we need due to Alfsen and Effros (see Proposition II.1.1 \cite{HWW}). Let $Z \subset Y$ be a closed subspace such that there is a linear
projection $P: Y^\ast \rightarrow Y^\ast$ such that $ker(P)= Z^\bot$ and $\|y^\ast\|= \|P(y^\ast)\|+\|y^\ast-P(y^\ast)\|$ for all $y^\ast \in Y^\ast$ (such
a subspace $Z$ is called a $M$-ideal in $Y$), then $Z$ is a proximinal subspace of $Y$. Over the years there have been several simpler proofs of this result, using ideas  substantially different from those in \cite{HWW}, Chapter 1. See the survey article \cite{BR} and \cite{R1} where we also gave a simple proof
of the result, $M$-ideals are proximinal and strongly proximinal subspaces.
\vskip 1em
In what follows we assume that $Y^\ast = L^1(\mu)$ for a positive measure $\mu$. When $\mu$ is non-atomic we need an additional condition on the position of ${\mathcal K}(X,Z)$. Let $Z \subset Y$ be a strongly proximinal  subspace of finite codimension. From Lemma 5, we assume that $k$ is the smallest integer
such that $Z^\bot \subset \ell^1(k)$.
\begin{thm}
Let $X$ be a reflexive Banach space and let $Y$ be a $L^1$-predual space. Let $Z \subset Y$ be a  finite codimensional strongly proximinal subspace. Suppose ${\mathcal K}(X,Z)^\bot \subset X^\ast {\otimes^{\vee}_{\pi} }\ell^1(k)$ . Then ${\mathcal K}(X,Z)$ is a proximinal subspace of ${\mathcal K}(X,Y)$.	
\end{thm}
\begin{proof}
It is easy to see that ${\mathcal K }(X,Z)$ is a subspace of finite codimension of ${\mathcal K}(X,Y)$. Since $X$ is reflexive and as $Y^\ast = L^1(\mu)$ has the metric approximation property, we have ${\mathcal K}(X,Y)^\ast = X^\ast  {\otimes_{\pi}^{\vee}} Y^\ast$.
\vskip 1em
Since $Z$ is a finite codimensional strongly proximinal subspace of $Y$, from Remark 4 and the comments preceding it, we get that $Z^\bot$ consists of $SSD$ points of $Y^\ast$. Therefore by Lemma 5, let $k$ be the smallest positive integer such that $Z^\bot \subset \ell^1(k)$ and $Y^\ast = \ell^1(k) \bigoplus_1 N$, for some closed subspace of $N \subset Y^\ast$.
\vskip 1em
Since $\ell^1$-sum is distributive for projective tensor products,
we have $X^\ast {\otimes_{\pi}^{\vee}} Y^\ast = X^\ast {\otimes_{\pi}^{\vee}} \ell^1(k) \bigoplus_1 X^\ast {\otimes_{\pi}^{\vee}}N$.
Since $X$ is reflexive, $X^\ast {\otimes_{\pi}^{\vee}} \ell^1(k)$ is reflexive and hence a weak$^\ast$-closed subspace of ${\mathcal K}(X,Y)^\ast$. Let $M = \{T \in {\mathcal K}(X,Y): \tau(T)=0~for~all~\tau \in X^\ast {\otimes_{\pi}^{\vee}} \ell^1(k)\}$.
\vskip 1em
As $M^\bot = X^\ast {\otimes^{\vee}_{\pi} }\ell^1(k)$, $M$ is a $M$-ideal in ${\mathcal K}(X,Y)$ and hence
 proximinal subspace of ${\mathcal K}(X,Y)$. We next show that $M \subset {\mathcal K}(X,Z) \subset {\mathcal K}(X,Y)$.
\vskip 1em
We will show that ${\mathcal K}(X,Z)^\bot  \subset M^\bot$.
  Now  by hypothesis ${\mathcal K}(X,Z)^\bot \subset M^\bot$ and therefore $M \subset {\mathcal K}(X,Z) \subset {\mathcal K}(X,Y)$.
\vskip 1em
Since $M$ is a proximinal subspace of ${\mathcal K}(X,Y)$ and ${\mathcal K}(X,Z)/M$ is a reflexive space, by Garkavi's theorem (\cite{S} Chapter 1, Theorem 2.1)
we get that ${\mathcal K}(X,Z)$ is a proximinal subspace of ${\mathcal K}(X,Y)$.
\end{proof}
We recall that for a closed subspace $Z \subset Y$ of finite codimension $n$, there exists independent unit vectors $f_1,...,f_n \in Z^\bot$
such that $Z = \bigcap_1^n Ker(f_i)$. In what follows we apply the preceding arguments to both $Y$ and $Y^{\ast\ast}$. Since $X$ is a reflexive space, for  $T \in {\mathcal K}(X,Y)$, $T^{\ast\ast} = T \in {\mathcal K}(X, Y)$. For a $L^1$-predual space $Y$, we also note that since all the duals of even higher order are spaces of continuous functions, the assumptions of the following theorem thus apply for all duals of higher even order of $Y$ (we always consider the canonical embeddings).
\begin{thm} Let $X$ , $Y$ satisfy all the  assumption made in the statement of  Theorem 6. Let $T \in {\mathcal K}(X,Y^{\ast\ast})$. $d(T, {\mathcal K}(X,Z^{\bot\bot}))= \|T -S\|$ for some $S
\in {\mathcal K}(X,Z^{\bot\bot})$.
\end{thm}
\begin{proof}
Since $Y$ is a $L^1$-predual space, $Y^{\ast\ast}$ is isometric to $C(K)$ for some compact Hausdorff space $K$ (see \cite{L}, Chapter 7).
We note that $Z^{\bot\bot} = \bigcap_1^n ker(f_i)$. We have $f_i \in Y^\ast$ and are $SSD$ points. Also since $Y^\ast = L^1(\mu)$, by Kakutani's Theorem (see \cite{L}, Chapter 1), $Y^{\ast\ast\ast} = C(K)^\ast = Y^\ast \bigoplus_1 N$ for some closed subspace $N \subset C(K)^\ast$. It is easy to see that $f_i$'s continue to be points of  $SSD$ in $C(K)^\ast$. Thus we again have $Z^{\bot\bot\bot} =Z^{\bot} \subset \ell^1(k)$ and $C(K)^\ast = \ell^1(k) \bigoplus_1 N'$
for some closed subspace $N' \subset C(K)^\ast$. Also as ${\mathcal K}(X,Z) \subset {\mathcal K}(X, Z^{\bot\bot})\subset {\mathcal K}(X,C(K))$, we get
that ${\mathcal K}(X, Z^{\bot\bot})^\bot \subset {\mathcal K}(X, Z)^\bot \subset  X^\ast {\otimes^{\vee}_{\pi} }\ell^1(k)$, as in the hypothesis of Theorem 6. Now proceeding exactly as in the proof of Theorem 6, we conclude that ${\mathcal K}(X, Z^{\bot\bot})$ is a proximinal subspace of ${\mathcal K}(X,C(K))$.
\end{proof}
\section{Proximinal subspaces of finite codimension in ${\mathcal K}(X,Y)$}
Let $X$ be a reflexive Banach space having the metric approximation property.
In this section we describe a class of Banach spaces $Y$ for which proximinal subspaces in spaces of operators can be determined using $SSD$ points.
\vskip 1em
Let $\{Y_i\}_{i \in I}$ be an infinite family of finite dimensional Banach spaces.
Let  $Y$ be a Banach space such that $Y^\ast = \bigoplus_1 Y_i$.
For notational simplicity, we omit writing the fixed infinite index set $I$. This in particular covers the case where $Y$ is a $L^1$-predual space where $\mu$ is a discrete measure.
\vskip 1em
 We recall that a closed subspace $Z \subset Y$ is said to be a factor reflexive space, if $Y/Z$ is a reflexive space.
\begin{thm}
Let $Z \subset {\mathcal K}(X,Y)$ be a closed subspace such that $Z^\bot$ consists of $SSD$ points. Then $Z$ is a factor reflexive, proximinal subspace of ${\mathcal K}(X,Y)$.	
\end{thm}
\begin{proof}
By our assumptions, we have ${\mathcal K}(X,Y)^\ast = X \otimes^{\wedge}_{\pi} Y^\ast$.
And $X \otimes^{\wedge}_{\pi} Y^\ast = \bigoplus_1( X \otimes^{\wedge}_{\pi}Y_i)$.  Further for every $i \in I$, since $Y_i$ is  finite dimensional,  $X \otimes^{\wedge}_{\pi}Y_i$ is a reflexive space.
\vskip 1em
Let $f \in Z^\bot$ be a $SSD$ point. We may and do assume that $f$ has only countably many non-zero coordinates. We note that the definition of a $SSD$ point does not involve the predual of the space .
\vskip 1em
Consider the $c_0$-direct sum,  $ V = \bigoplus_{c_0}(X ^\ast\otimes^{\vee}_{\epsilon}Y^\ast_i)$. We have $V^\ast = \bigoplus_1 X \otimes^{\wedge}_{\pi}Y_i =Y^\ast$.
\vskip 1em
Since $f$ is a $SSD$ point, it attains its norm on $V$. Since $V$ is an infinite $c_0$-sum, it is easy to see that $f$ has at most finitely many non-zero coordinates.
\vskip 1em
Since $Z^{\bot}$ is a Banach space, by the Baire category theorem, it is easy to see that for a finite set $A \subset I$, $f(i) = 0$ for all $i \notin A$ and $f \in Z^{\bot}$.
\vskip 1em
Thus we have the decomposition $X \otimes^{\wedge}_{\pi} Y^\ast = \bigoplus _1 \{X \otimes^{\wedge}_{\pi} Y_i: i \in A\} \bigoplus_i\{X \otimes^{\wedge}_{\pi}Y_i: i \notin A\}$.
And $ Z^\bot \subset \bigoplus _1 \{X \otimes^{\wedge}_{\pi} Y_i: i \in A\} $. Since $Y_i$'s are finite dimensional and $A$ is a finite set, this latter space is a reflexive space and hence weak$^\ast$-closed space. Therefore $Z$ is a factor reflexive space.
\vskip 1em
As before let $M = \{ T \in {\mathcal K}(X,Y): \tau(T) = 0~for~all~\tau
\in \bigoplus _1 \{X \otimes^{\wedge}_{\pi} Y_i: i \in A\} \}$. Now $M$ is a $M$-ideal and hence proximinal subspace. Also $M \subset Z$. Therefore by an application of Garkavi's theorem again, we see that $Z$ is proximinal in ${\mathcal K}(X,Y)$.
\end{proof}
\begin{cor} Let $\{Y_i\}_{i \in I}$ be an infinite family of finite dimensional space.
Let $Y$ be a Banach space such that $Y^\ast = \bigoplus_1 Y_i$. If $Z \subset Y$ is a closed subspace such that $Z^\bot$ consists of $\mathnormal{SSD}$ points, then $Z$ is a finite codimensional strongly proximinal subspace of $Y$.
\end{cor}
\begin{proof}
	Proceeding as in the proof of the above theorem, we see that $Z$ is a finite codimensional space. Since $M$-ideals are strongly proximinal, by using Remark 4, we see that $Z$ is a strongly proximinal subspace of $Y$.
\end{proof}
In the case $Y = \bigoplus_{c_0} Y_i^\ast$ for the family $\{Y_i\}_{i \in I}$ of the finite dimensional spaces considered above, we have proximinality
result for ${\mathcal L}(X,Y)$, when $X$ is a reflexive Banach space with the approximation property.
\begin{thm} Let $X$ and $Y$ be as above. Let $Z \subset Y$ be a proximinal subspace of finite codimension. ${\mathcal K}(X,Z)$ is a proximinal subspace of
${\mathcal L}(X,Y)$.
\end{thm}
\begin{proof}
To keep the arguments simple, we assume that $I$ is the set of positive integers.
Let $Z = \cap_1^k ker(\alpha_i)$ for some $\{\alpha_i\}_{1 \leq i \leq k} \subset \bigoplus_1 Y_i$. Since $Z$ is a proximinal subspace, using Garkavi's theorem, we have that $ker(\alpha_i)$ is a proximinal subspace, for all $i$. Hence $\alpha_i$'s attain their norm on $Y$.
	
	\vskip 1em
	
	By arguments indicated before, we may assume that there is a $n>0$, $\alpha_i(j)=0$
	for $j>n$ and $1 \leq i \leq k$.
	
	If $M = \bigoplus_{c_0}\{Y^\ast_i:~i> n\}$,
	then $Z = F \bigoplus_{\infty} M$, for some finite dimensional space $F \subset \bigoplus_{\infty} \{Y_i^\ast: 1 \leq i \leq n\}$.
	\vskip 1em
	It is easy to see, since $X$ has the approximation property,  that ${\mathcal K}(X,Z) =
	{\mathcal L}(X, F)\bigoplus_{\infty} {\mathcal K}(X,M)$.
	It is also easy to see that ${\mathcal K}(X,M)$ can be identified with $\bigoplus_{c_0}\{X^\ast \otimes_{\epsilon}^{\vee} Y_i^\ast\}$ where the  sum is taken over $i>n$ . We note that the component spaces, $X^\ast \otimes_{\epsilon}^{\vee} Y_i^\ast$ are reflexive.
	
	\vskip 1em
	We further note that ${\mathcal L}(X, F)$ is a reflexive space.
	
	\vskip 1em
	Since $Y = \bigoplus_{\infty}\{Y^\ast_i: 1 \leq i \leq n\} \bigoplus_{\infty} M$, by similar reasoning, we have, ${\mathcal L}(X,Y)= {\mathcal L}(X, \bigoplus_{\infty}\{Y^\ast_i: 1 \leq i \leq n\})
	\bigoplus_{\infty}  {\mathcal L}(X,M)$.
	\vskip 1em
We now compare the components in the direct sums in ${\mathcal K}(X,Z)$ and ${\mathcal L}(X,Y)$.
The first component of ${\mathcal K}(X,Z)$, being a reflexive space is a proximinal subspace of the first component of ${\mathcal L}(X,Y)$.
\vskip 1em
We note that $T \rightarrow T^\ast$ maps ${\mathcal L}(X,M)$ isometrically into ${\mathcal L}(M^\ast, X^\ast) = (M^\ast \otimes_{\pi}^{\wedge} X)^\ast $.
\vskip .5em
Since projective tensor product is distributive over $\ell^1$-sums, we have $M^\ast \otimes_{\pi}^{\wedge} X = \bigoplus_1(Y_i \otimes_{\pi}^{\wedge} X)$.
Since $Y_i$'s are finite dimensional, we thus have ${\mathcal L}(M^\ast,X^\ast) = \bigoplus_{\infty}(X^\ast \otimes_{\epsilon}^{\vee} Y_i)$.
Here the sums in both cases are taken over $i>n$.
\vskip 1em
It is well known that $c_0$-sum of a family of spaces is always a proximinal subspace of $\ell^\infty$-sum of the same collection (see the proof of Proposition II.1 in \cite{HWW}).
Since ${\mathcal K}(X,M) \subset {\mathcal L}(X,M)\subset \bigoplus_{\infty}(X^\ast \otimes_{\epsilon}^{\vee} Y_i)$, we conclude that
${\mathcal K}(X,M)$ is a proximinal subspace of ${\mathcal L}(X,M)$.
\vskip .5em
Since the component spaces are $\ell^\infty$-sums, we get the conclusion ${\mathcal K}(X,Z)$ is a proximinal subspace of ${\mathcal L}(X,Y)$.
\end{proof}

\end{document}